\documentclass[12pt]{amsart}

\usepackage{tgpagella}
\usepackage{euler}
\usepackage[T1]{fontenc}
\usepackage{amsmath, amssymb}
\usepackage[hidelinks]{hyperref}
\usepackage[english]{babel}
\usepackage{mathrsfs}
\usepackage{eucal}
\usepackage[all]{xy}

\newtheorem{thm}{Theorem}[subsection]
\newtheorem*{thm*}{Theorem}
\newtheorem{lem}[thm]{Lemma}
\newtheorem{fact}[thm]{Fact}

\newtheorem{prop}[thm]{Proposition}
\newtheorem*{prop*}{Proposition}

\newtheorem{cor}[thm]{Corollary}

\theoremstyle{definition}

\newtheorem{remark}[thm]{Remark}

\newtheorem{question}[thm]{Question}

\def\cal{\mathcal}

\newcommand\cU{{\cal U}}

\newcommand\bbN{{\mathbf N}}

\begin{document}


\title{On the nonexistence of F\o lner sets}
\author{Isaac Goldbring}
\thanks{Goldbring's work was partially supported by NSF CAREER grant DMS-1349399.}

\address {Department of Mathematics, University of California, Irvine, 340 Rowland Hall (Bldg.\# 400), Irvine, CA, 92697-3875.}
\email{isaac@math.uci.edu}
\urladdr{http://www.math.uci.edu/~isaac}

\begin{abstract}
We show that there is $n\in \mathbf N$, a finite system $\Sigma(\vec x,\vec y)$ of equations and inequations having a solution in some group, where $\vec x$ has length $n$, and $\epsilon>0$ such that:  for any group $G$ and any $\vec a\in G^n$, if the system $\Sigma(\vec a,\vec y)$ has a solution in $G$, then there is \emph{no} $(\vec a,\epsilon)$-F\o lner set in $G$.  The proof uses ideas from model-theoretic forcing together with the observation that no amenable group can be existentially closed.  Along the way, we also observe that no existentially closed group can be exact, have the Haagerup property, or have property (T).  Finally, we show that, for $n$ large enough and for $\epsilon$ small enough, the existence of $(F,\epsilon)$-F\o lner sets, where $F$ has size at most $n$, cannot be expressed in a first-order way uniformly in all groups.
\end{abstract}

\maketitle

\section{Introduction}

Recall that, for $G$ a group, a finite subset $F$ of $G$, and $\epsilon>0$, a finite subset $H$ of $G$ is called a \textbf{$(F,\epsilon)$-F\o lner set} if, for all $g\in F$, we have that $|gH\triangle H|<\epsilon |H|$.  If $\vec a=(a_1,\ldots,a_n)$ is an $n$-tuple from $G$, we refer to $(\{a_1,\ldots,a_n\},\epsilon)$-F\o lner subsets of $G$ as $(\vec a,\epsilon)$-F\o lner subsets of $G$.  The group $G$ is said to be \textbf{amenable} if, for every finite subset $F$ of $G$ and every $\epsilon>0$, there is a $(F,\epsilon)$-F\o lnet subset of $G$.


The following is the main result of our note:

\begin{thm}
There is $n\in \mathbf N$, a finite system $\Sigma(\vec x,\vec y)$ of equations and inequations having a solution in some group, where $\vec x$ has length $n$, and $\epsilon>0$ such that:  for any group $G$ and any $\vec a\in G^n$, if the system $\Sigma(\vec a,\vec y)$ has a solution in $G$, then there is \emph{no} $(\vec a,\epsilon)$-F\o lner set in $G$.
\end{thm}

We will actually prove something a bit more general, namely that the set of such systems is dense in a certain sense to be made precise below.

It is natural to ask whether or not such a system could actually be equivalent (in all groups) to the nonexistence of a F\o lner set.  We will remark that, \emph{asymptotically} (meaning for $n$ large enough and $\epsilon$ small enough), no such system can exist.

The proof of the main theorem follows in a rather straightforward way from the observation that no \emph{existentially closed (e.c.) group} can be amenable together with some standard facts from \emph{model-theoretic forcing}.  For the sake of the reader not familiar with these concepts, we will review them below.  We will however assume that the reader is familiar with some basic logic; an introduction aimed towards group theorists can be found in \cite{hodges}.

We also take the opportunity to observe that e.c.\ groups cannot have some of the other mainstream properties studied in group theory nowadays, namely exactness, the Haagerup property, or property (T).  We also take up the corresponding question for so-called \emph{locally universal} groups and settle this question for some of these properties.

\section{Existentially closed groups are not amenable}

In the rest of this note, $L$ denotes the language of groups and $T$ denotes the $L$-theory of groups.

\subsection{Preliminaries on existentially closed and locally universal groups}

Recall that a group $G$ is said to be \textbf{existentially closed (e.c.)} if, for any finite system $p(\vec x)$ of equations and inequations with coefficients from $G$, if there is a solution to $p(\vec x)$ in an extension of $G$, then there is a solution to $p(\vec x)$ in $G$.  (In model-theoretic terms:  if there is an existential $L$-formula $\varphi(\vec x)$ with parameters from $G$ that is satisfiable in an extension of $G$, then $\varphi(\vec x)$ is satisfiable in $G$ itself.)    

It is useful to rephrase being e.c.\ in the following way:  

\begin{fact}\label{ecembed}
A group $G$ is e.c.\ if and only if:  whenever $G$ is a subgroup of $H$, then there is an ultrapower $G^\cU$ of $G$ and an embedding $i:H\hookrightarrow G^\cU$ such that $i|G$ is the diagonal embedding of $G$ into $G^\cU$. 
\end{fact}

We will need the following facts about e.c.\ groups; see, for example, \cite[Theorem 1.8a and Theorem 5.8]{higmanscott}.

\begin{fact}\label{ecfacts}

\

\begin{enumerate}
\item E.c.\ groups are not finitely generated.
\item Every finitely generated group with solvable word problem embeds into every e.c.\ group.
\end{enumerate}
\end{fact}

Fact \ref{ecembed} implies that e.c.\ groups are \textbf{locally universal}, where a group $G$ is said to be locally universal if every group embeds into some ultrapower of $G$.\footnote{To see this, suppose that $G$ is e.c.\ and $H$ is an arbitrary group.  Then since $G\subseteq G\times H$, it follows that $G\times H$ embeds into an ultrapower of $G$, whence so does $H$.} 

Below we will need the following lemma, which is the analog of Fact \ref{ecfacts}(2) for locally universal groups.  The proof is nearly identical to the aforementioned fact, but we include it here for the sake of the reader.

\begin{lem}\label{locuniv}
Every finitely generated group with solvable word problem embeds into every locally universal group.
%
%
\end{lem}

\begin{proof}
Suppose that $G$ is a finitely generated group with solvable word problem and that $H$ is a locally universal group; we show that $G$ embeds into $H$.  By a celebrated theorem of Higman and Boone \cite[Theorem 7.4]{LS}, there is a finitely presented group $K$ with a simple subgroup $S$ such that $G$ embeds into $S$.  It thus suffices to show that $S$ embeds into $H$.  Let $K=(\vec a \ | \ \vec w)$ be a finite presentation of $K$, where $\vec a=(a_1,\ldots,a_m)$ and $\vec w=(w_1,\ldots,w_n)$.  Fix $s\in S\setminus \{e\}$ and let $w(x_1,\ldots,x_m)$ be a word such that $s=w(a_1,\ldots,a_m)$.  Let $\sigma$ be the sentence
$$\exists \vec x\left(\bigwedge_{i=1}^n w_i(\vec x)=e \wedge w(\vec x)\not=e\right).$$  Since $H$ is locally universal, $K$ embeds into $H^\cU$, whence $H^\cU\models \sigma$ and thus $H\models \sigma$.  Take $b_1,\ldots,b_n\in H$ that witness the truth of $\sigma$.  Then the mapping $a_i\mapsto b_i$ yields a homomorphism $K\to H$ whose restriction to $S$ is not identically the identity of $H$.  Since $S$ is simple, we have that the aformentioned map restricts to an embedding of $S$ into $H$, as desired. 
\end{proof}

As a special case of the previous lemma, we have that every simple finitely presented group embeds into every locally universal group; this simpler statement does not need the aformentioned result of Higman and Boone and follows from a simpler version of the previous proof.

Notice that the previous lemma seemingly yields a generalization of Fact \ref{ecfacts}(2).  However, it is not clear to us if Lemma \ref{locuniv} truly is a strengthening of Fact \ref{ecfacts}(2):

\begin{question}
Does every locally universal group contain an e.c.\ group?
\end{question}

It is known that no e.c.\ group embeds into a finitely presented group (\cite[Corollary 6.10]{higmanscott}); thus, in connection with the previous question, it is natural to ask:

\begin{question}
Can a locally universal group every be finitely presented?
\end{question}

\subsection{Some properties that existentially closed groups never have}

The following proposition is central to the main theorem of this note:

\begin{prop}
If $G$ is an e.c.\ group, then $G$ is not amenable.
\end{prop}

\begin{proof}
By Fact \ref{ecfacts}(2) and the fact that $\mathbb F_2$ has solvable word problem (\cite[Corollary 1.3]{LS}), $\mathbb F_2$ embeds into $G$, whence $G$ is not amenable (\cite[Corollary G.3.5]{Tbook}).
\end{proof}

Although it will not be necessary for our main theorem, it is interesting to note the extension of the previous proposition\footnote{In some sense, this proposition shows that the discrete version of the Connes Embedding Problem is false.}:

\begin{cor}
If $G$ is locally universal, then $G$ is not amenable.
\end{cor}

\begin{proof}
Argue just as in the previous proof, this time applying Lemma \ref{locuniv}.
\end{proof}

\begin{remark}
If one prefers not to use the above theorem of Higman and Boone in the previous proof, one can instead use the fact that simple, finitely presented, nonamenable groups exist, e.g. Thompson's group $T$ \cite{Thompson} or the Burger-Mozes groups \cite{BM}.
\end{remark}

\begin{remark}
Since $\mathbb F_2$ is exact, the previous proof shows that no amenable group can even be locally universal for the class of exact groups.\footnote{We thank David Kerr for asking us if the content of this remark was true.}
\end{remark}

Speaking of exactness:

\begin{prop}
If $G$ is a locally universal group, then $G$ is not exact.
\end{prop}

\begin{proof}
By Fact \ref{locuniv} together with the fact that exactness is preserved under subgroup, it suffices to find a non-exact finitely generated group with solvable word problem.  Such groups exist, e.g. the Gromov monster \cite{gromov}.
\end{proof}

%
%

Another generalization of amenability, the \emph{Haagerup property}, can also never be a property of a locally universal group:

\begin{prop}
If $G$ is a locally universal group, then $G$ does not have the Haagerup property.
\end{prop}

\begin{proof}
Again, by Fact \ref{locuniv} together with the fact that the Haagerup property is preserved under subgroup, it suffices to find a finitely generated group with solvable word problem that does not have the Haagerup property.  For example, $\operatorname{SL}_3(\mathbb Z)$ is a finitely generated group with solvable word problem that has property (T) (\cite[Theorem 1.4.15]{Tbook}), whence cannot have the Haagerup property (\cite{cornulier}).
\end{proof}


Speaking of property (T):

\begin{prop}
If $G$ is an e.c.\ group, then $G$ does not have property (T).
\end{prop}

\begin{proof}
Groups with property (T) are finitely generated (\cite[Theorem 1.3.1]{Tbook}), whence we can conclude by referring to Fact \ref{ecfacts}(1).
\end{proof}

\begin{question}
Can there be a locally universal group with property T?
\end{question}

We end this subsection by remarking that if a single locally universal group is sofic, then all groups are sofic.  Indeed, it is easily verified that the (discrete) ultrapower of a sofic group is sofic; since subgroups of sofic groups are sofic, the observation follows.  In particular, if a single e.c.\ group is sofic, then all groups are sofic.\footnote{This latter statement was observed by Glebsky in \cite{glebsky}, although with a more complicated proof.}

\subsection{Existence of F\o lner sets is an $\forall\bigvee\exists$-property}
Fix $m,n\in \bbN$ and $\epsilon>0$ and set $\vec x:=(x_1,\ldots,x_n)$.  In what follows, we write $I\subseteq_\epsilon [m]$ to mean that $I\subseteq \{1,\ldots,m\}$ and $|I|\geq (1-\epsilon)m$.  We set
$\varphi_{m,n,\epsilon}(\vec x)$ to be the $L$-formula
$$\exists y_1\cdots y_m\left(\bigwedge_{j\not=k}y_j\not=y_k\wedge \bigwedge_{i=1}^n \bigvee_{I\subseteq_\epsilon [m]}\bigwedge_{j\in I}\bigvee_{k=1}^m x_iy_j=y_k\right).$$
Set $\varphi_{n,\epsilon}(\vec x):=\bigvee_m \varphi_{m,n,\epsilon}(\vec x)$ and $\sigma_{n,\epsilon}:=\forall \vec x\varphi_{n,\epsilon}(\vec x)$.

The proof of the following proposition is clear:

\begin{prop}\label{AVE}
Let $G$ be a group and $\vec a\in G^n$.
\begin{enumerate}
\item $G\models \varphi_{m,n,\epsilon}(\vec a)$ if and only if there is an $(\vec a,\epsilon)$-F\o lner subset of $G$ of size $m$.
\item $G\models \varphi_{n,\epsilon}(\vec a)$ if and only if there is an $(\vec a,\epsilon)$-F\o lner subset of $G$.
\item $G\models \sigma_{n,\epsilon}$ if and only if, for every subset $F$ of $G$ of size at most $n$, there is an $(F,\epsilon)$-F\o lner subset of $G$.
\item $G\models \bigwedge_{n,\epsilon}\sigma_{n,\epsilon}$ if and only if $G$ is amenable.  (Here, it suffices to assume that $\epsilon$ ranges over rational numbers.)
\end{enumerate}
\end{prop}

The sentences $\sigma_{n,\epsilon}$ are prototypical examples of $\forall\bigvee\exists$-sentences, which are special kind of $L_{\omega_1,\omega}$-sentences\footnote{$L_{\omega_1,\omega}$ is the extension of first-order logic that allows countable conjunctions and disjunctions provided only finitely many free variables appear in all of the formulae involved in the conjunction or disjunction.}. It is known that if P is a $\forall\bigvee\exists$-property and there is a locally universal object with property P, then there is an e.c.\ object with property P.  (For a proof, see, for example, \cite[Proposition 2.6]{enforce}, although this is in the context of continuous logic.)  The previous proposition, together with the fact that no e.c.\  group can be amenable, gives a different proof of the fact that no locally universal group can be amenable.

\section{Model-theoretic forcing and the main result}

\subsection{Preliminaries on model-theoretic forcing}  In this subsection, we outline the idea of model-theoretic forcing (restricted to the case of groups).  Our approach follows that of Hodges \cite{hodges}.

We fix a countably infinite set $C$.  A \textbf{condition} is a finite set $p(\vec c)$ of equations and inequations in the variables $\vec c$ from $C$.

We consider a two-player game $\mathcal G$ with $\omega$ many rounds defined as follows.  The players take turns playing conditions, each time ensuring that the condition being played extends the condition played by the previous player.  The outcome of the game is an infinite chain 
$$p_0\subseteq p_1\subseteq p_2\subseteq \cdots$$ of conditions whose union we denote by $\mathbf p$.  We let $\mathbf p_e$ denote the set of equations appearing in $p$.  We set $G_p:=(C\ |\ \mathbf p_e)$, that is, $G_\mathbf p$ is the group generated by the set $C$ and whose relations are the equations that are played at some point in the game.  One refers to $G_\mathbf p$ as the \textbf{compiled group} resulting from the play $\mathbf p$ of $\mathcal G$.

Let $P$ be a property of groups.  We say that $P$ is \textbf{enforceable} if the second player has a strategy for $\mathcal G$ such that, if the second player follows that strategy, then the compiled group will have property $P$.

The followings facts can be found in \cite{hodges} as Corollary 3.4.3 and Lemma 2.3.3(e) respectively.
\begin{fact}\label{forcefacts}

\

\begin{enumerate}
\item It is enforceable that the compiled group be e.c.
\item (Conjunction lemma) If, for each $n\in \mathbf N$, $P_n$ is an enforceable property, then the conjunction $\bigwedge_n P_n$ is also enforceable.
\end{enumerate}
\end{fact}

If $\sigma$ is a sentence of $\mathcal L_{\omega_1,\omega}$, we write $\Vdash \sigma$ to mean that the property ``the compiled group believes that $\sigma$ is true'' is enforceable.

If $p$ is a condition, we can consider the game $\mathcal G_p$ defined as above except that the first player is required to play a condition extending $p$.  We say that $p$ \textbf{forces} $P$ if the second player has a strategy for $\mathcal G_p$ such that, if the second player follows that strategy, then the compiled group will have property $P$.  As before, if $\sigma$ is a sentence of $\mathcal L_{\omega_1,\omega}$, we write $p\Vdash \sigma$ to mean that $p$ forces the property ``the compiled group believes that $\sigma$ is true.''  We will need the easy observation that if $\Vdash \neg \sigma$, then for all conditions $p$, we have that $p\not\Vdash \sigma$.\footnote{In fact, the converse is also true.}

\subsection{Proof of the main result}

\begin{thm}\label{notforce}
There is $n\in \mathbf N$ and $\epsilon>0$ such that $\not\Vdash \sigma_{n,\epsilon}$.
\end{thm}

\begin{proof}
If $\Vdash \sigma_{n,\epsilon}$ held for all $n\in \mathbf N$ and all $\epsilon>0$, then, by Fact \ref{forcefacts}(2), amenability would be enforceable.  By Fact \ref{forcefacts}(1), being e.c.\ is also enforceable, whence it follows that there would be an amenable e.c.\  group, yielding a contradiction.
\end{proof}

We need the following fundamental result from model-theoretic forcing:

\begin{fact}\label{fundfact}
Let $p$ be a condition, and, for each $n\in \mathbf N$, let $\Phi_n(\vec x)$ be an existential $L$-formula.  Then $p\not\Vdash \forall \vec x \bigvee_{n\in \bbN} \Phi_n(\vec x)$ if and only if there is an existential $L$-formula $\psi(\vec x)$ such that $T\cup p\cup\{\exists \vec x\psi(\vec x)\}$ is consistent and such that $$T\models\forall \vec x\left(\psi(\vec x)\rightarrow \bigwedge_{n\in \bbN}\neg\Phi_n(\vec x)\right).$$
\end{fact}

\begin{proof}
This is nearly the statement of \cite[Theorem 3.4.4]{hodges} (stated for an arbitrary theory), except that the $\Phi_n$'s are allowed to contain constants from $C$, whence so is $\psi$.  In case that the $\Phi_n$'s are actually $L$-formulae, we can assume that $\psi$ is also an $L$-formula by existentially quantifying out any mention of elements from $C$.
\end{proof}

The following is the main theorem announced in the introduction; it follows immediately from Theorem \ref{notforce} and Fact \ref{fundfact}:

\begin{thm}
There is an $n\in \mathbf N$, $\epsilon>0$, and an existential $L$-formula $\psi(\vec x)$, where $\vec x=(x_1,\ldots,x_n)$, such that $T\models \forall \vec x(\psi(\vec x)\rightarrow \neg \varphi_{n,\epsilon}(\vec x))$.
\end{thm}

\begin{remark}
Let $\psi(\vec x)$ be as in the statement of the previous theorem.  Then $\psi(\vec x)$ has infinitely many realizations in every e.c.\ group, whence we have ``concrete'' witnesses to the fact that e.c.\ groups are not amenable.
\end{remark}

We can in fact find ``densely'' many such $\psi(\vec x)$.  First, one more fundamental forcing fact, which is a special case of \cite[Theorem 3.4.7]{hodges}:

\begin{fact}
For any $L_{\omega_1,\omega}$-sentence $\sigma$, we have that either $\Vdash \sigma$ or $\Vdash \neg \sigma$.
\end{fact}

\begin{cor}
There is $n\in \mathbf N$ and $\epsilon>0$ such that $\Vdash \neg\sigma_{n,\epsilon}$.  For such $n$ and $\epsilon$ and for any condition $p$, we have that $p\not\Vdash \sigma_{n,\epsilon}$.
\end{cor}

Let us recap in purely group theoretic terms:

\begin{cor}
There is $n\in \mathbf N$ and $\epsilon>0$ such that:  for any system $\Sigma_0(\vec w)$ of equations and inequations in the finite set of variables $\vec w$, there is a system $\Sigma(\vec w,\vec z,\vec x)$ of equations and inequations, where $\vec w$, $\vec z$, and $\vec x$ are finite disjoint sets of variables and $\vec x$ has length $n$, such that $\Sigma_0\subseteq \Sigma$ and such that, for any group $G$ and any $\vec a\in G^n$ if $\Sigma(\vec w,\vec z,\vec a)$ has a solution in $G$, then there is no $(\vec a,\epsilon)$-F\o lner subset of $G$.
\end{cor}

\subsection{On the asymptotic undefinability of the existence of F\o lner sets}

We say that a group $G$ is \textbf{uniformly amenable} if there is a function $\Phi:\mathbf N\times (0,1)\to \mathbf N$ such that, given any subset $F$ of $G$ of size at most $n$ and any $\epsilon\in (0,1)$, there is a $(F,\epsilon)$-F\o lner subset of $G$ of size at most $\Phi(n,\epsilon)$.

There are amenable groups that are not uniformly amenable, e.g. $S_\infty:=\bigcup_n S_n$.\footnote{We learned of this fact from the online lecture notes \cite{ross} of David Ross.}

\begin{thm}
There is $n_0\in \mathbf N$ and $\epsilon_0>0$ such that, for any $n\geq n_0$ and any $\epsilon\leq \epsilon_0$, there can be no first-order formula equivalent in all groups to $\neg \varphi_{n,\epsilon}(\vec x)$.  
\end{thm}

\begin{proof}
Let $G$ be an amenable group that is not uniformly amenable as witnessed by $n_0\in \mathbf N$ and $\epsilon_0>0$.  Suppose $n\geq n_0$ and $\epsilon\leq \epsilon_0$ and that $T\models \forall \vec x(\psi(\vec x)\rightarrow \neg\varphi_{n,\epsilon}(\vec x))$.  Then the set of formulae
$$\{\neg \varphi_{m,n,\epsilon}(\vec x) \ : \ m\in \bbN\}\cup\{\neg \psi(\vec x)\}$$ is finitely satisfiable in $G$ (since $G$ is amenable).  Thus, there is $\vec a\in (G^\cU)^n$ such that $\neg \varphi_{n,\epsilon}(\vec a)$ holds while $\psi(\vec a)$ fails.
\end{proof}


\begin{thebibliography}{99}
\bibitem{Tbook} B. Bekka, P. de la Harpe, and A. Valette, \textit{Kazhdan's property (T)}, New Mathematical Monographs \textbf{11}, Cambridge University Press, Cambridge, 2008.
\bibitem{BM} M. Burger and S. Mozes, \textit{Finitely presented simple groups and products of trees}, C.R. Acad. Sci. Paris, t. 324, Seríe I (1997).
\bibitem{Thompson} J.W. Cannon and W.J. Floyd, \textit{What is...Thompson's group?}, Notices of the AMS \textbf{58-8} (2011), 1112-1113. 
\bibitem{cornulier} Y. de Cornulier, \textit{Kazhdan and Haagerup properties in algebraic
groups over local fields}, J. Lie Theory \textbf{16} (2006), 67-82.
\bibitem{glebsky} L. Glebsky, \textit{Approximation of groups, characterizations of sofic groups, and equations over groups}, J.Algebra \textbf{477} (2017), 147-162.
\bibitem{enforce} I. Goldbring, \textit{Enforceable operator algebras}.  arXiv 1706.09048.
\bibitem{gromov} M. Gromov, \textit{Random walk in random groups}, Geom. Funct. Anal. \textbf{13} (2003), 73-146.
\bibitem{higmanscott} G. Higman and E. Scott, \textit{Existentially closed groups}, London Mathematical Society Monographs (1988).
\bibitem{hodges} W. Hodges, \textit{Building models by games}, London Mathematical Society Student Texts \textbf{2}, Cambridge University Press, Cambridge, 1985.
\bibitem{LS} R. Lyndon and P. Schupp, \textit{Combinatorial group theory}, Classics in Mathematics, Springer (2001).
\bibitem{ross} D. Ross, \textit{Lecture notes on nonstandard analysis and groups}, online lecture notes available at \texttt{http://www.math.hawaii.edu/~ross/L-seminar10-07.pdf}

\end{thebibliography}
\end{document}